\documentclass{article}
\usepackage{graphicx}
\usepackage{amsfonts}
\usepackage{amsmath}
\usepackage{amssymb}
\usepackage{url}
\usepackage{fancyhdr}
\usepackage{indentfirst}
\usepackage{enumerate}
\usepackage{amsthm}
\usepackage{color}
\usepackage{natbib}

\def\d{\mathrm{d}}

\def\lawcn{\buildrel d \over \rightarrow}
\newcommand{\var}{\mathrm{Var}}

\newcommand{\VaR}{\mathrm{VaR}}

\newcommand{\TVaR}{\mathrm{ES}}
\newcommand{\E}{\mathbb{E}}
\newcommand{\R}{\mathbb{R}}
\newcommand{\N}{\mathbb{N}}
\newcommand{\p}{\mathbb{P}}

\newcommand{\id}{\mathrm{I}}
\newcommand{\lcx}{\prec_{\mathrm{cx}}}

\renewcommand{\ge}{\geqslant}
\renewcommand{\le}{\leqslant}

\renewcommand{\epsilon}{\varepsilon}
\theoremstyle{plain}
\newtheorem{theorem}{Theorem}
\newtheorem{corollary}[theorem]{Corollary}
\newtheorem{lemma}[theorem]{Lemma}
\newtheorem{proposition}[theorem]{Proposition}
\theoremstyle{definition}
\newtheorem{definition}{Definition}[section]
\newtheorem{example}{Example}[section]

\theoremstyle{remark}
\newtheorem{remark}{Remark}[section]

\numberwithin{equation}{section} \numberwithin{theorem}{section}

\begin{document}

\title{Extreme Negative Dependence and Risk Aggregation}

\author{Bin Wang\thanks{Department of Mathematics, Beijing Technology and Business University, Beijing 100048, China.}~\  and Ruodu Wang\thanks{Corresponding author.
Department of Statistics and Actuarial Science, University of Waterloo.
(email: \texttt{wang@uwaterloo.ca}).}
}
\date{\today}
\maketitle

\begin{abstract}  We introduce the concept of an extremely negatively dependent (END) sequence of random variables with a given common marginal distribution. The END structure, as a new benchmark for negative dependence, is comparable to comonotonicity and independence. We show that an END sequence always exists for any given marginal distributions with a finite mean and we provide a probabilistic construction. Through such a construction, the partial sum of identically distributed but dependent random variables is controlled by a random variable that  depends only on the marginal distribution of the sequence. The new concept and derived results are used to obtain asymptotic bounds for risk aggregation with dependence uncertainty.

\begin{bfseries}Key-words\end{bfseries}: central limit theorem; variance reduction; sums of random variables; dependence uncertainty; risk aggregation.


\end{abstract}

\section{Introduction\label{intro}}

For a given univariate distribution (function) $F$ with finite mean $\mu$, let $X_1,X_2,\ldots$ be any sequence of random variables from the distribution $F$ and denote the partial sum $S_n=X_1+\cdots+X_n$ for $n\in \N$. The distribution of $S_n$ varies under different assumptions of dependency (joint distribution) among the sequence $(X_i, i\in \N)$. For example, if we assume that the variance of $F$ is finite, then it is well-known that
\begin{itemize}
\item[\textbf{(a)}] if $X_1,X_2,\ldots$ are independent, $S_n$ has a variance of order $n$, and $(S_n-n\mu)/\sqrt n$ converges weakly to a normal distribution  (Central Limit Theorem);
\item[\textbf{(b)}] if $X_1,X_2,\ldots$ are \emph{comonotonic} (when $X_1,X_2,\ldots$ are identically distributed, this means $X_1=\cdots=X_n$ a.s.), $S_n$ has a variance of order $n^2$ and $S_n/n$ is always distributed as $F$.
\end{itemize}
However,  the following question remains: among all possible dependencies, is there one dependency which gives the following \textbf{(c1)} or \textbf{(c2)}?
\begin{itemize}
\item[\textbf{(c1)}] $S_n, n\in \N$ have variance bounded by a constant. Equivalently, $S_n$ has a variance of order $O(1)$ as $n\rightarrow \infty$;
\item[\textbf{(c2)}] $(S_n-n\mu)/k_n$ converges a.s. for any $k_n\rightarrow \infty$ as $n\rightarrow \infty$. It is easy to see that this limit has to be zero.
\end{itemize}

  In this paper, we answer questions \textbf{(c1)}-\textbf{(c2)}.  In contrary to the positive dependence in \textbf{(b)}, we use the term \emph{extremely negative dependence} (END) for the dependence scenario which gives  \textbf{(c2)}.  We show that there is always an END that yields \textbf{(c2)}, and the same dependency also gives \textbf{(c1)} if we further assume the third moment of $F$ is finite. Within our framework \textbf{(c1)} is stronger since it at least requires a finite variance and \textbf{(c2)} always has a positive answer, although \textbf{(c1)} and \textbf{(c2)} are not comparable for a general sequence.
 Moreover, we show that there exists a dependency among random variables $X_1,X_2,\ldots$ such that $|S_n-n\mu|$ is controlled by a single random variable $Z$, the distribution of which  is in terms of $F$ and does not depend on $n$.

The research on questions of the above type is very much related to the following fundamental question:
\begin{itemize}
\item[\textbf{(A)}] what are the possible distributions of the random variable $S_n$ without knowing the dependence structure of $(X_1,\ldots,X_n)$?
\end{itemize}
Here, theoretically, $S_n$ can be replaced by any functional of $(X_1,\ldots,X_n)$. In this paper we focus on $S_n$ for it is the most typical functional studied in the literature, and it has self-evident interpretations in applied fields.
Question \textbf{(A)} is a typical question concerning uncertain dependence structures of random vectors. It involves optimization over functional spaces with non-linear constraints, and
 is closely related to research on copula theory, mass transportation theory, Monte-Carlo (MC) and Quasi-MC (QMC) simulation, and quantitative risk management.  The interested reader is referred to \cite{N06} (for copula theory), \cite{RR98} (for mass transportation),  \cite{G06} (for (Q)MC simulation) and \cite{MFE05} (for quantitative risk management). Moreover, in \cite{R13} (Parts I and II), these links as well as recent research on them are extensively discussed  with a perspective of financial risk analysis.

Question \textbf{(A)}  turns out to be highly non-trivial. As far as we know, even for the case $n=2$, \textbf{(A)} is still open.  In the literature, a weaker version of \textbf{(A)} is studied more often:
\begin{itemize}
\item[\textbf{(B)}] what are the \emph{extremal} (in some sense) distributions of $S_n$ without knowing the dependence structure of $(X_1,\ldots,X_n)$?
\end{itemize}
Of course,  here we need to define the term \emph{extremal} mathematically. Different definitions lead to different solutions and different approaches.
The first answer to a question of type \textbf{(B)} was given by \cite{M81},
who, in response to a question earlier raised by
A.N. Kolmogorov, gave the maximal and minimal values of the distribution function of $S_2=X_1+X_2$ for given marginal distributions of $X_1$ and $X_2$. \cite{R82} independently gave the answer to the same question based on a different approach originating from mass transportation theory.

Finding the point-wise minimal and maximal values of the distribution $F_S$ of $S_2$  does not directly imply all possible distributions in \textbf{(A)}; a global characterization is unavailable.  Hence, the question \textbf{(A)} is only partially answered. Unfortunately, even if we limit the discussion to minimum and maximum of the distribution function, Marakov's and R\"uschendorf's methods cannot be extended into $n\ge 3$ cases without assuming specific forms of $F$. More recently, a series of papers \cite{EP06}, \cite{EPR13} and \cite{WPY13} discussed the minimal and maximal questions under different assumptions on $F$.

Another direction of research related to \textbf{(B)} considers the worst-case variance or expected convex functions of $S_n$. In that case, the criterion for the extremal distribution can for instance be chosen as
\begin{equation}\label{eq1s1}\max\{\var(S_n):X_i\sim F, ~i=1,\ldots,n\},\end{equation}
or
\begin{equation}\label{eq2s1}\max\{\E[(S_n-x)_+]:X_i\sim F, ~i=1,\ldots,n\}, ~~~x \in \R,\end{equation}
where  $(\cdot)_+=\max\{\cdot,0\}.$
These criteria lead to the well-known {comonotonic scenario} \textbf{(b)}, where $X_1=\cdots=X_n$ is the solution to the above optimization problems. The comonotonic scenario conveniently gives the maximum convex ordering element among all possible dependence structures (for definition and properties of the convex order, see \cite{SS07}, Chapter 1). For comonotonicity and its applications, the reader is referred to \cite{DDV11}.

On the other hand, answering the minimal questions seems to be more challenging. For example, the analytical solutions to the optimization problems \begin{equation}\label{eq3s1}\min\{\var(S_n):X_i\sim F, ~i=1,\ldots,n\},\end{equation}
and
\begin{equation}\label{eq4s1}\min\{\E[(S_n-x)_+]:X_i\sim F, ~i=1,\ldots,n\},  ~~~x \in \R,\end{equation}
are unknown for general marginal distributions when $n\ge3$; \eqref{eq3s1}-\eqref{eq4s1} may indeed lead to different optimal dependence structures. For a disucssion on the convex ordering minimal elements with given marginals, we refer to \cite{BJW13} and the references therein. Questions \eqref{eq3s1}-\eqref{eq4s1} are typical variance reduction problems, and hence they naturally apply to the sample generating procedure in MC simulations. See  \cite{RU02} and \cite{WW11} for recent developments on the explicit solutions to   \eqref{eq3s1}-\eqref{eq4s1}.  It is obvious that the questions  \textbf{(c1)}-\textbf{(c2)},  in an asymptotic manner, are directly linked to the optimization problems \eqref{eq3s1}-\eqref{eq4s1}.

Questions  \textbf{(c1)}-\textbf{(c2)} are also relevant to the study of \emph{risk aggregation with dependence uncertainty} (see for example, \cite{BJW13}) in quantitative risk management. The aggregate position $S_n$ represents the total risk or loss random variable in a given period, where $X_1,\ldots,X_n$ are individual risk random variables. Assume we know the marginal distributions of $X_1,\ldots,X_n$ but the joint distribution of $(X_1,\ldots,X_n)$ is unknown.  This assumption is not uncommon in risk management where interdependency modeling  relies very heavily on data and computational resources. A risk regulator or manager may for instance be interested in a particular risk measure $\rho$ of $S_n$. However, without information on  the dependence structure,   $\rho(S_n)$ cannot be calculated. It is then important to identify the extreme cases: the largest and smallest possible values of $\rho(S_n)$, and this relates to question \textbf{(B)} and in particular,  to \textbf{(c1)}-\textbf{(c2)}. To obtain the extreme values of $\rho(S_n)$ for finite $n$, a strong condition of  \emph{complete mixability}is usually imposed in the literature, and explicit values are only available for some specific choices of marginal distributions; see for example \cite{WW11}, \cite{WPY13}, \cite{EPR13} and \cite{BJW13}. On the other hand, there is limited research on the asymptotic behavior of $\rho(S_n)$ as $n\rightarrow \infty$. In this paper, we use the concept of END to derive asymptotic estimates for the popular risk measures $\VaR$ and $\TVaR$ of $S_n$ as $n\rightarrow \infty$ for any marginal distribution $F$. As a consequence, our results based on END lead to the asymptotic equivalence between worst-case VaR and ES, shown recently by \cite{PR13} and \cite{PWW13} under different assumptions on $F$. As an improvement, our result does not require any non-trivial conditions on $F$, and gives the convergence rate of this asympotic equivalence.

The rest of the paper is organized as follows.
In Section 2, we study the sum of END random variables, and show that   the sum is controlled by a random variable with distribution derived from $F$. Some examples are given and a link between complete mixability and END is provided. In Section 3, we provide asymptotic bounds for  expected convex functions and risk measures of the aggregate risk with dependence uncertainty, and we further establish an asymptotic equivalence between the worst-case VaR and worst-case ES. Section 4 draws some conclusions.
In this paper, we assume that all random variables that we discuss in this paper are defined on a common general atomless probability space $(\Omega, \mathcal A, \p).$  In such a probability space, we can generate independent random vectors with any distribution.

\section{Extremely Negatively Dependent Sequence}

Throughout the paper, we denote $S_n=X_1+\cdots+X_n$ where $X_1,\ldots,X_n$ are random variables with distribution $F$, if not specified otherwise, and we assume that the mean $\mu$ of $F$ is finite.
We also define the generalized inverse function of any distribution function $F$ by $F^{-1}(t)=\inf\{x:F(x)\ge t\}$ for $t\in (0,1]$ and its left endpoint $F^{-1}(0)=\inf\{x:F(x)> 0\}$.

\subsection{Main results}

In this section we will show that there exists a sequence of random variables $X_1,X_2,\ldots$ with common distribution $F$, such that   $|S_n-n\mu|$ is controlled by a random variable $Z$ that does not depend on $n.$ It turns out that such a random variable $Z$ has a distribution  derived directly from $F$. We call it the \emph{residual distribution} of $F$, as defined below.

 The idea behind is that we try to construct a sequence of random variables $X_1,X_2,\ldots$  such that each of the member compensates the sum $S_n$. For each random variable $X_i$, we consider two possibilities: $X_i$ is ``large" and $X_i$ is ``small". We design a dependence  such that the number of ``large" $X_i$'s, $i=1,\cdots,n$ and the number of ``small" $X_i$'s $i=1,\cdots,n$ are balanced in a specific way. Moreover, the ``large" part and the 	``small" part are counter-monotonic so that they compensate each other. We first introduce some notation.

  Let $$H(s)=\int_0^s (F^{-1}(t)-\mu) \d t, ~~ s\in[0,1],$$
and  denote $\nu=F(\mu-)$ and $\nu^+=F(\mu)$ (when $F$ does not have a probability mass at $\mu$, $\nu=\nu^+$).
It is easy to see that the  function $H$ is bounded, strictly decreasing on $[0,\nu]$, strictly increasing on $[\nu^+,1]$, $H(0)=H(1)=0$, and the minimum value of $H(s)$ is attained at $c:=H(\nu)=H(\nu^+)<0$. Moreover, $H$ is a convex function and hence is almost everywhere (a.e.) differentiable on $[0,1]$. For each $s\in (c,0]$, denote by $A(s)\in [0,\nu)$ and $B(s)\in (\nu^+,1]$ such that $H(A(s))=s$ and $H(B(s))=s$, i.e. $A$ and $B$ are the inverse functions of $H$ on the two intervals $[0,\nu)$ and $(\nu^+,1]$, respectively. Moreover, let $A(c)=B(c)=\nu^+$. Note that since $H$ has an a.e. non-zero derivative, $A$ and $B$ are a.e. differentiable on $[c,0]$.
Let $$K(s)=\left\{\begin{array}{cc} 1& s>0,\\B(s)-A(s) &c<s\le 0,\\\nu^+-\nu&s=c ,\\0 &s< c.\end{array}\right.$$ $K(s)$ is right-continuous, increasing, $K(c-)=0$, and  $K(0)=1$, hence it is a distribution function on $[c,0]$ with probability mass $\nu^+-\nu$ at $c$ and $K$ is continuous on $(c,0]$.
Note that $H$, $A$, $B$, and $K$ all depend on $F$. Later, we will see that $B$ leads to the ``large" values of $X_i$ and  $A$ leads to the ``small" values of $X_i$.

\begin{definition}\label{def1}
The \emph{residual distribution} of  a distribution $F$ is the distribution of the random variable $F^{-1}(B(Y))-F^{-1}(A(Y))$, where $Y\sim K$. The residual distribution of a distribution $F$ is denoted by $\tilde F$.
\end{definition}

Using this definition, we are able to present our first result.

\begin{theorem}\label{main1a}
Suppose $F$ is a distribution with mean $\mu$,  then there exist $X_i\sim F$, $i\in \N$ and $Z\sim \tilde F$, such that for each $n\in \N$,
\begin{equation}|S_n-n\mu|\le Z.\label{main1eq1a}\end{equation}
\end{theorem}

\begin{proof}
We prove this theorem by construction. Define
$$u(s)=\frac{\mu-F^{-1}(A(s))}{F^{-1}(B(s))-F^{-1}(A(s))}, ~~~s\in (c,0);$$
and in addition we let $u(c)=1/2$.
It is easy to see that $u(s) \in [0,1]$ for $s\in [c,0).$ The following lemma contains a key step in the construction of the   sequence $X_i,~i\in\N$.
\begin{lemma}\label{lem0}
Suppose $F$ is a distribution with mean $\mu$. Let $Y$ be a random variable with distribution $K$ and $U$ be a $\mathrm U[0,1]$ random variable, independent of $Y$. Let
$$
X= A(Y)\id_{\{U\ge u(Y)\}}+B(Y)\id_{\{U< u(Y)\}},
$$
Then $F^{-1}(X)\sim F.$
\end{lemma}
\begin{proof}[Proof of Lemma \ref{lem0}]
Note that $A(Y)< \nu$ and $B(Y)> \nu^+$ when $Y\ne c.$ Hence, the possible values of $X$ are divided into three subsets: $\{X\in [0,\nu)\}={\{U\ge u(Y)\}}\cap \{Y\ne c\}$, $\{X=\nu^+\}=\{Y=c\}$, $\{X\in (\nu^+,1]\}={\{U< u(Y)\}}\cap \{Y\ne c\}$.
For $t\in[0,\nu),$
\begin{align*}
\p(X\le t)&= \p(A(Y)\le t, U\ge u(Y), Y\ne c)\\
&= \p(Y\ge H(t), U\ge u(Y))\\
&= \int_{H(t)}^0 (1-u(y))\d K(y)\\
&= \int_{H(t)}^0 \frac{F^{-1}(B(y))-\mu}{F^{-1}(B(y))-F^{-1}(A(y))}\d (B(y)-A(y)).
\end{align*}
Since $B$ and $A$ are the inverse functions of $H$, we have a.e.
$$\d B(y)=\frac{1}{H'(B(y))}\d y=\frac{1}{F^{-1}(B(y))-\mu}\d y,$$ and $$\d A(y)=\frac{1}{H'(A(y))}\d y=\frac{1}{F^{-1}(A(y))-\mu}\d y.$$ Thus
\begin{align*}
\p(X\le t)&=\int_{H(t)}^0 \frac{F^{-1}(B(y))-\mu}{F^{-1}(B(y))-F^{-1}(A(y))} \frac{F^{-1}(A(y))-F^{-1}((B(y))}{(F^{-1}(B(y))-\mu)(F^{-1}(A(y))-\mu)}\d y\\
&=\int_{H(t)}^0 \frac{1}{\mu-F^{-1}(A(y))}\d y\\
&=\int_{t}^0 \frac{1}{\mu-F^{-1}(s)}(F^{-1}(s)-\mu)\d s\\
&=t.
\end{align*}
Similarly, we can  show that $\p(X> t)=1-t$ for $t\in (\nu^+,1]$. Hence, there exists a random variable $U\sim \mathrm{U}[0,1]$ such that $X=U$ when $U\in [0,\nu)\cup(\nu^+,1].$ It is also easy to see that, when $U\in [\nu,\nu^+],$ we have $X=\nu^+$ and $F^{-1}(U)=\mu=F^{-1}(\nu^+)=F^{-1}(X)$. In conclusion, $F^{-1}(X)=F^{-1}(U)$ a.s. and thus $F^{-1}(X)\sim F$.
\end{proof}

We continue to prove Theorem \ref{main1a}. Let $Y$ be a random variable with distribution $K$ and $U$ be a $\mathrm U[0,1]$ random variable independent of $Y$. For $k\in \N$, define
$$Y_k=A(Y)\id_{\{\underline{U+k u(Y)} \ge u(Y)\}}+B(Y)\id_{\{\underline{U+k u(Y)}< u(Y)\}},$$
and $$X_k=F^{-1}(Y_k),$$
where $\underline x: =x-\lfloor x \rfloor$ is the fractional part of a real number $x$. It is easy to see that $\underline{U+k u(Y)}$  is U$[0,1]$ distributed and is independent of $Y$. Hence, by Lemma \ref{lem0} we know that  $X_k\sim F$, $k\in \N$.

An intuition of this construction is as follows. Denote $W_1=F^{-1}(B(Y))$ and $W_2=F^{-1}(A(Y))$.
As we can see, there are two possibilities for the random variable $X_k$: it is either   $W_1$ (roughly speaking, representing large values of $X_k$) or $W_2$ (representing small values of $X_k$). Note that $u(Y)W_1+(1-u(Y))W_2=\mu$. By constructing random variables $X_k,~ k\in \N$ in this specific way, we aim to let $W_1$ and $W_2$ compensate each other, leading to an $S_n$ that is close to its mean. In the following we complete the proof.

Denote $C_k=\{\underline{U+k u(Y)} < u(Y)\}$ for $k\in \N$. It is easy to see that
$$\id_{C_k}=\#\{N\in \N: N\in (U+(k-1)u(Y),U+ku(Y)]\}.$$
Thus, for $n\in \N$,
 $$\sum_{i=1}^n \id_{C_i}=\#\{N\in \N: N\in (U,U+nu(Y)]\}=\# (\N\cap  (U,U+nu(Y)]).$$
It follows that
$$\lfloor nu(Y)\rfloor \le \sum_{i=1}^n \id_{C_i}\le \lfloor nu(Y)\rfloor+1.$$
 We have that, when $\sum_{i=1}^n \id_{C_i}=\lfloor nu(Y)\rfloor$, or equivalently $U<1-\underline{nu(Y)}$,
\begin{align} S_n&=W_1\sum_{i=1}^n \id_{C_i} +W_2\left(n- \sum_{i=1}^n \id_{C_i}\right) \nonumber \\
&=\lfloor nu(Y)\rfloor W_1+(n-\lfloor nu(Y)\rfloor)W_2
\nonumber\\&=n\left(u(Y) W_1+(1-u(Y))W_2\right)-\underline{nu(Y)}\left(W_1-W_2\right)
\nonumber\\&=n\mu-\underline{nu(Y)}\left(W_1-W_2\right), \label{bounded1}
\end{align}
and  when $\sum_{i=1}^n \id_{C_i}=\lfloor nu(Y)\rfloor+1$, or equivalently $U\ge 1-\underline{nu(Y)}$, that
\begin{align}
S_n&=n\mu-\underline{nu(Y)}\left(W_1-W_2\right)+W_1-W_2\nonumber\\&=n\mu+(1-\underline{nu(Y)})\left(W_1-W_2\right).\label{bounded2} \end{align}
By \eqref{bounded1}-\eqref{bounded2}, we have
\begin{align}
S_n-n\mu=\left(W_1-W_2\right) (\id_{\{U\ge 1-\underline{nu(Y)}\}}-\underline{nu(Y)}). \label{bounded2a}
\end{align}
Thus, we obtain
$
|S_n-n\mu|\le W_1-W_2,
$ and by definition $W_1-W_2=F^{-1}(B(Y))-F^{-1}(A(Y))\sim \tilde F.$
\end{proof}

\begin{remark}
If $F$ does not have a probability mass at $\mu$, $X$ in Lemma \ref{lem0} is U$[0,1]$ distributed, and $Y$ is a continuous random variable on $[c,0]$.
\end{remark}

\begin{remark}
From the proof of Theorem \ref{main1a}, we can see that for $n>m,$ $S_n-S_m=\sum_{i=m}^n X_i$ also satisfies $|S_n-S_m-(n-m)\mu|\le Z.$
In the above proof, the sigma field of $(X_i,~i\in \N)$ is generated by two independent random variables $U$ and $Y$.
\end{remark}

 With the results in Theorem \ref{main1a}, we can answer questions \textbf{(c1)}-\textbf{(c2)} regarding the extremely negative dependence. First, we give a formal definition of END. Recall the two questions given in the introduction: \begin{itemize}
\item[\textbf{(c1)}] $S_n,~ n\in \N$ have variance bounded by a constant;
\item[\textbf{(c2)}] $(S_n-n\mu)/k_n \rightarrow 0$ a.s. for any $k_n\rightarrow \infty$ as $n\rightarrow \infty$.
\end{itemize}
\begin{definition} Consider a sequence of random variables $(X_i, ~i\in \N)$  with common distribution $F$.
We say that $(X_i, ~i\in \N)$ is \emph{extremely negatively dependent} (END), if \textbf{(c2)} holds. Moreover,
we say that $(X_i, ~i\in \N)$ is \emph{strongly extremely negatively dependent} (SEND), if \textbf{(c1)}-\textbf{(c2)}  hold and $$\sup_{n\in \N}\var(S_n)\le \sup_{n\in \N}\var(Y_1+\cdots+Y_n)$$ for any sequence of random variables $(Y_i,~i\in \N)$ with common distribution $F$.
\end{definition}
As discussed in the introduction, the END structure is the opposite to comonotonicity.  The SEND structure can be treated as the most negative correlation between random variables in a sequence, and hence serves as a potential candidate in variance minimization problems and MC simulations. Also note that any finite number of random variables in a sequence does not affect the property of END but they do affect the property of SEND.
\begin{remark}
The criterion of minimizing $\sup_{n\in N}\var(S_n)$ in the definition of an  SEND  sequence can be replaced by another optimization criterion, such as $\sup_{n\in N}\E[g(S_n)]$  or $\limsup_{n\rightarrow\infty}\E[g(S_n)]$ for a convex function $g$. The reason why we choose the variance as the criterion is that it gives a comparison with the classic Central Limit Theorem, and also meets the interests of variance reduction in applied fields.
\end{remark}

Using Theorem \ref{main1a}, we have the following immediate corollary. It gives general bounds for the sum $S_n$ and the existence of an END structure.

\begin{corollary}\label{coro1}  Suppose $F$ is a distribution with mean $\mu$.
\begin{enumerate}[(a)]\item
If the support of $F$ is contained in $[a,b],$ $a,b\in \R$,  then there exist $X_i\sim F$, $i\in \N$ such that for each $n\in \N$,
\begin{equation}|S_n-n\mu|\le b-a.\label{main1eq1}\end{equation}
 \item There exist $X_i\sim F$, $i\in \N$ such that $(S_n-n\mu)/k_n \rightarrow 0$ a.s. for any $k_n\rightarrow \infty$ as $n\rightarrow \infty$.
In other words, there exists an END sequence of random variables from $F$.
\end{enumerate}
\end{corollary}
 \begin{remark}
The above result shows that the sequence of probability measures generated by the sequence $S_n-n\mu, n\in \N$ in
Corollary \ref{coro1} is \emph{tight} (see for example, \cite{B99}, Chapter 1).
\end{remark}


One may wonder about the relationship between $F$ and $\tilde F$. The following lemma gives a link between the moments of both distribution functions.

\begin{lemma}\label{lem2}
If $F$ has finite $k$-th moment, $k>1$, then $\tilde F$ has finite $(k-1)$-st moment.
\end{lemma}
\begin{proof}Without loss of generality, we assume $\mu=0$. We use the notation $W_1$ and $W_2$ as in the proof of Lemma \ref{lem2}. Note that  by definition,  $W_1\ge0$, $W_2\le 0$, and $\E[\min\{W_1,|W_2|\}=0$ if and only if $W_1=0=W_2$ (the lemma holds trivially in this case). In the following we assume $\E[\min\{W_1,|W_2|\}]>0$.

By \eqref{bounded2a}, setting $n=1$, we have \begin{align}\label{momenteq2}\nonumber &\E[|X_1-\mu|^k]\\&=\E[|S_1-\mu|^k]
\nonumber\\&=\E[(W_1-W_2)^k |\id_{\{U\ge 1- u(Y)\}}- u(Y)|^k]
\nonumber\\
&=\E[(W_1-W_2)^k \E[|\id_{\{U\ge 1- u(Y)\}}- u(Y)|^k|Y]]
\nonumber\\&=\E\left[(W_1-W_2)^k u(Y)( 1-u(Y)) \left((1-u(Y))^{(k-1)}+u(Y)^{(k-1)}\right)\right]
\nonumber\\&=\E\left[(W_1-W_2)^k \frac{-W_2}{W_1-W_2}\frac{W_1}{W_1-W_2}\left(\left(\frac{W_1}{W_1-W_2}\right)^{(k-1)}+\left(\frac{-W_2}{W_1-W_2}\right)^{(k-1)}\right)\id_{\{W_2\ne W_1\}}\right]
\nonumber\\
&=\E\left[\frac{-W_2W_1}{W_1-W_2}(W_1^{k-1}+(-W_2)^{k-1})\id_{\{W_2\ne W_1\}}\right]
\nonumber\\
&\ge \E\left[\frac{-W_2W_1}{W_1-W_2}(\max\{W_1, |W_2|\})^{k-1}\id_{\{W_1>0\}}\right]
\nonumber\\
&\ge \E\left[\frac{\max\{W_1, |W_2|\} \min\{W_1, |W_2|\}}{2\max\{W_1, |W_2|\})}(\max\{W_1, |W_2|\})^{k-1}\id_{\{W_1>0\}}\right]
\nonumber\\&=\E\left[\frac{1}{2}\min\{W_1, |W_2|\}(\max\{W_1, |W_2|\})^{k-1}\id_{\{W_1>0\}}\right]
\nonumber\\&= \E\left[\frac{1}{2}\min\{W_1, |W_2|\}(\max\{W_1, |W_2|\})^{k-1}\right].
\end{align}
Since $\E[|X_1-\mu|^k]$ is finite, $\E\left[\min\{W_1, |W_2|\}(\max\{W_1, |W_2|\})^{k-1}\right]$ is finite.   Note that $W_1$ and $|W_2|$ are comonotonic by definition, hence \begin{equation}\label{momenteq}\E\left[\min\{W_1, |W_2|\}(\max\{W_1, |W_2|\})^{k-1}\right]\ge \E[\min\{W_1, |W_2|\}] \E[\max\{W_1, |W_2|\})^{k-1}].\end{equation} Recall that we assume $\E[\min\{W_1, |W_2|\}]>0$, and hence $\E[(\max\{W_1,|W_2|\})^{k-1}]<\infty$ follows from \eqref{momenteq2}-\eqref{momenteq}.   Finally, by definition, $W_1-W_2$ has  distribution $\tilde F$, thus $\tilde F$ has finite $(k-1)$-st moment.
\end{proof}

\begin{remark}\label{remark2.5}
From \eqref{momenteq2}, we can see that if the distribution functions of $W_1$ and $|W_2|$ are asymptotically equivalent (i.e. $\p(W_1>x)/\p(|W_2|>x)=O(1)$ and $\p(|W_2|>x)/\p(W_1>x)=O(1)$ as $x \rightarrow \infty$), then the finiteness of  the $k$-th moment of $F$ actually  implies the finiteness  of  the $k$-th moment of $\tilde F$. When one of $W_1$ and $W_2$ is bounded but the other one is unbounded, only the  finiteness  of  the $(k-1)$-st moment of $\tilde F$ is guaranteed. The relation \eqref{momenteq2} is sharp in the sense that the two inequalities used in \eqref{momenteq2} are tight inequalities which at most reduce the quantity by three fourths.
\end{remark}
\begin{proposition}\label{main3} Suppose $F$ is a distribution with mean $\mu$, and $F$ has finite $m$-th moment, $m>1$. Then there exist $X_i\sim F$, $i\in \N$ such that uniformly in $n\in \N$, as $k\rightarrow \infty$,
\begin{equation}\label{unbounded3}\p(|S_n-n\mu| >k )=
o(k^{-{m+1}}).\end{equation}
In particular,  as $n\rightarrow \infty$, for all $\epsilon>0$,
\begin{equation}\label{unbounded4}\p(|S_n-n\mu| <n^{\epsilon} )=1-o(n^{-(m-1)\epsilon}).\end{equation}
\end{proposition}
\begin{proof}
The finiteness of the $(m-1)$-th moment of $\tilde F$ guarantees that $x^{m-1} (1-\tilde F(x))\rightarrow 0$ as $x\rightarrow \infty$.  Hence, by Theorem \ref{main1a}, we have that $\p(|S_n-n\mu| >k )\le 1-\tilde F(k)=
o(k^{-{m+1}})$.\end{proof}

To seek for a possible SEND sequence, we present a link between the variances of $F$ and $\tilde F$.
\begin{proposition} \label{propvar} Suppose $\tilde F$ has finite variance. Then there exist $X_i\sim F,~ i\in\N$ and $Z \sim \tilde F$  such that for $n\in \N$,
$$\var(S_n)\le \frac14\E[Z^2].$$
In particular, for such $X_i\sim F,~ i\in\N$, we have that
\begin{enumerate}[(a)]
\item  $\var(S_n)\le  (b-a)^2/4 $ if $F$ is supported on $[a,b]$, $a,b\in \R$;
\item  $\var(S_n)\le C$ for some constant $C$ that does not depend on $n$  if $F$ has finite third moment, and
\item the sequence  $X_i\sim F,~ i\in\N$ is SEND if $\var(X_1)=\E[Z^2]/4.$
\end{enumerate}
\end{proposition}
\begin{proof} We use the notation $W_1$ and $W_2$  as in the proof of Lemma \ref{lem2}, and let $Z=W_1-W_2$. By \eqref{bounded2a},
\begin{align*}\var(S_n)&=\E[(S_n-n\mu)^2]
\\&=\E[(W_1-W_2)^2 (\id_{\{U\ge 1-\underline{nu(Y)}\}}-\underline{nu(Y)})^2]
\\
&=\E[(W_1-W_2)^2 \E[(\id_{\{U\ge 1-\underline{nu(Y)}\}}-\underline{nu(Y)})^2|Y]]
\\&=\E[(W_1-W_2)^2\underline{nu(Y)}( 1-\underline{nu(Y)})]
\\&\le \frac14\E[(W_1-W_2)^2].
\end{align*}
The results follow from this:
\begin{enumerate}[(a)]
\item  This can be seen from the fact that $0\le Z=W_2-W_1\le |b-a|$.
\item
By Lemma \ref{lem2},  when $F$ has finite third moment, $Z\sim \tilde F$ has finite second moment. Thus, $\var(S_n)\le  \E[Z^2]/4<C. $
\item For any sequence $Y_i\sim F,~ i\in\N$, $$\sup_{n\in \N}\var(Y_1+\cdots+Y_n)\ge \var(Y_1)=\frac14\E[Z^2]\ge \sup_{n\in \N}\var(S_n).$$ Hence, $X_i\sim F,~ i\in\N$ are SEND.
\end{enumerate}
\end{proof}
Finding sequences of random variables  with small total variance (such as the END sequence) is a classical question  in
variance reduction and simulation (see for example \cite{F72}).  It is
 especially important in Monte-Carlo (MC) and Quasi Monte-Carlo (QMC) simulation (for instance, see \cite{G06} for (Q)MC methods and their applications in finance), where typically a dependence structure is chosen to generate a random sample such that the error $|S_n/n-\mu|$ is approximately $a/{\sqrt{n}}$ with a small value of $a$. QMC techniques, such as low-discrepancy methods, aim for an error of order $O(n^{-(1-\epsilon)})~,\epsilon>0,$ by choosing (usually deterministic) discretization points. In our paper, we give a dependence structure which generates a random sample with an asymptotic error  of order $O(1/n)$ which significantly improves the convergence rate. Of course, the details of possible new random sample generation techinques, as well as the setup for high-dimensionality, need further research.

Yet, when $\var(X_1)<\E[Z^2]/4,$ it remains unclear to find an SEND sequence. From the examples in the next section, we would say that the bound $\E[Z^2]/4$ already gives  good estimates of the smallest variance of $S_n$ in general.

We conclude this section by a final remark on the variance of $S_n$ under the three different dependencies.
 As long as the third moment of $F$ is finite,
\begin{itemize}
\item if $X_i, ~i\in \N$ are independent, $\var(S_n)=O(n)$, and $(S_n-n\mu)/\sqrt n\lawcn $Normal;
\item if $X_i, ~i\in \N$ are comonotonic, $\var(S_n)=O(n^2)$ and $S_n/n\buildrel{\mathrm{a.s.}} \over = X_1\sim F$;
\item if $X_i, ~i\in \N$ are END, $\var(S_n)=O(1)$ and $(S_n-n\mu)/n^{\epsilon} \buildrel{\mathrm{a.s.,}~ L_2} \over{\longrightarrow} 0$ for any $\epsilon>0$.
\end{itemize}

\subsection{Examples}

In this section we give some examples of distributions $F$ and their residual distributions $\tilde F$. These examples show that some of the bounds given in Section 2.1 are sharp in the most general sense.
\begin{example}\label{ex1}
Suppose $F$ is a Bernoulli distribution on $\{0,1\}$ with parameter $p\in (0,1)$: $$F(x)=(1-p)\id_{\{x\ge 0\}}+p \id_{\{x\ge 1\}}.$$Then
$$H(s)=-ps\id_{\{0\le s\le 1-p\}}-(1-p)(1-s)\id_{\{1-p<s\le 1\}},~~~ s\in[0,1],$$
$H(s)$ attains its minimum at $H(1-p)=-p(1-p),$ and
$$A(t)=-\frac{t}{p}, ~~~B(t)=1+\frac{t}{1-p}, ~~~t\in [-p(1-p),0].$$
Therefore, $F^{-1}(B(t))=1$, $F^{-1}(A(t))=0$ for all $t \in (-p(1-p),0)$. This leads to $F^{-1}(B(Y))-F^{-1}(A(Y))=1$ a.s. Thus, $\tilde F$ is a degenerate distribution at 1,
and there exists a sequence of $X_1,X_2,\ldots$ with common distribution $F$ such that  for all $n\in \N$,
\begin{equation}\label{ex1b} |S_n-np|\le 1.\end{equation}
\end{example}

 \begin{remark} We consider two special cases of the above example. \begin{enumerate}\item The bound \eqref{ex1b} cannot be improved for an irrational $p$. Suppose $p$ in the above example  is an irrational number and let $X_1, X_2, \ldots$ be a sequence of random variables from $F$. It is obvious that $S_n$ is an integer, and $np$ is an irrational number. Since $\E[S_n]=np$, we must have $\p(S_n\le \lfloor np\rfloor)>0$ and  $\p(S_n\ge \lfloor np\rfloor+1)>0$. Thus, $\p(|S_n-np|\ge \underline{np})>0.$ Since $\{\underline{np}:n\in\N\}$ is dense in $[0,1]$, we have that for any $q\in[0,1)$, there are infinitely many $n\in\N$ such that $\p(|S_n-np|\ge q)>0.$ Hence, $|S_n-np|\le q$ with $q<1$ for all $n\in \N$ is impossible. This also confirms that for a distribution on $[a,b]$, the bound \eqref{main1eq1} given in Corollary \ref{coro1} (a) is sharp in general.
\item  When $p=1/2$, we have $\var(X_1)=1/4=\E[Z^2]/4$ where $Z\sim \tilde F$.  That is, the inequality in Proposition \ref{propvar} is an equality for $n=1$. By Proposition \ref{propvar} (c), the sequence $X_1,X_2,\ldots$ is SEND in the case of $p=1/2$.
\end{enumerate}
\end{remark}

\begin{example}
Suppose $F$ is a uniform distribution on $[0,1]$. Then
$$H(s)=\frac{1}{2}s(s-1),~~~ s\in[0,1],$$
$H(s)$ attains its minimum at $H(1/2)=-1/8,$ and
$$A(t)=\frac{1-\sqrt{1-8t}}{2}, ~~~B(t)=\frac{1+\sqrt{1-8t}}{2} , ~~~t\in \left[-\frac18,0\right].$$
Note that $Z:=F^{-1}(B(Y))-F^{-1}(A(Y))=B(Y)-A(Y),$ and $Y$ is a continuous random variable with distribution function $B-A$. It follows that $Z$ is U$[0,1]$ distributed. Therefore, $\tilde F$ is U$[0,1]$ and
and there exists a sequence of $X_1,X_2,\ldots$ with common distribution $F$ such that  for all $n\in \N$,
$$|S_n-n/2|\le Z.$$
\end{example}

\begin{remark} In the above example, $\var(X_1)=1/12$ and $\E[Z^2]=1/3.$  By Proposition \ref{propvar} (c), the sequence $X_1,X_2,\ldots$ is SEND.
\end{remark}

\begin{example}
Suppose $F$ is a Pareto distribution with index $\alpha=2$:
$$F(x)=1-x^{-2}, ~~x\ge 1.$$
Then $F^{-1}(t)=(1-t)^{-1/2},~ t\in (0,1),$ and $\mu=2$.
$$H(s)=2(1-s)-2\sqrt{1-s},~~~ s\in[0,1],$$
$H(s)$ attains its minimum at $H(3/4)=-1/2,$ and
$$A(t)=\frac{1-t-\sqrt{1+2t}}{2}, ~~~B(t)=\frac{1-t+\sqrt{1+2t}}{2} , ~~~t\in \left[-\frac12,0\right].$$
Note that for  $t\in \left[-\frac12,0\right]$
\begin{align}\nonumber F^{-1}(B(t))-F^{-1}(A(t))&=\frac{\sqrt2}{(1+t-\sqrt{1+2t})^{1/2}}-\frac{\sqrt2}{(1+t+\sqrt{1+2t})^{1/2}},
\\&=\frac{\sqrt 2((1+t+\sqrt{1+2t})^{1/2}-(1+t-\sqrt{1+2t})^{1/2})}{-t}
,\label{Y1}\end{align}
and
$$B(t)-A(t)=\sqrt{1+2t}.$$
$Y$ is a continuous random variable with distribution function $B-A$, and hence the inverse distribution function of $Y$ is $({s^2-1})/{2}, ~s\in (0,1)$. Thus we can write $Y=({U^2-1})/{2}$ where $U$ is a U$[0,1]$ random variable. Plugging it in \eqref{Y1}, we get
\begin{align*}Z:=F^{-1}(B(Y))-F^{-1}(A(Y))&=\frac{\sqrt 2((1+\frac{U^2-1}2+U)^{1/2}-(1+\frac{U^2-1}2-U)^{1/2})}{\frac{1-U^2}{2}}
\\&=\frac{2((U^2+1+2U)^{1/2}-(U^2+1-2U)^{1/2})}{{1-U^2}}
\\&=\frac{4U}{1-U^2}.\end{align*}
It follows that $$\tilde F(x)=\p(Z\le x)=\sqrt{1+\frac 4{x^2}}-\frac{2}{x},~~~ x\ge 0,$$ and
the tail of $\tilde F$ is Pareto-type with index 1.
There exists a sequence of $X_1,X_2,\ldots$ with common distribution $F$ such that  for all $n\in \N$,
$$|S_n-2n|\le Z.$$
Note that $Z\le 2/({1-U})$, hence $|S_n-2n|$ is controlled by another Pareto random variable, $2/(1-U)$, with index 1.
\end{example}

\begin{remark} In the above example, $F$ has finite $(2-\epsilon)$-th moment for all $\epsilon>0$, and $\tilde F$ has finite $(1-\epsilon)$-th moment for all $\epsilon>0$. This confirms the sharpness of the moment relation in Lemma \ref{lem2} for one-side bounded distributions (see Remark \ref{remark2.5}).
\end{remark}

\subsection{Extreme negative dependence and complete mixability}

There are of course more  ways to construct END sequences. One of them is through the idea of  completely mixability (CM), which is linked to  a ``perfect" negative dependence structure.

\begin{definition}[\cite{WW11}]
A univariate distribution $F$ is $n$-\emph{completely mixable} ($n$-CM) if  there exist  $X_i\sim F$, $i=1,\ldots,n$ such that $X_1+\cdots+X_n$ is a constant (or a.s. a constant). The vector $(X_1,\ldots,X_n)$ is called a \emph{complete mix}.
\end{definition}

Some straightforward examples and properties of CM distributions can be found in~\cite{WW11}  and \cite{PWW12}. For a fixed $n$, a complete mix is usually regarded as having the most negative correlation, in the sense that $\var(S_n)=0$. There are at least three major differences between the complete mix and the END sequence in Theorem \ref{main1a}.
\begin{enumerate}[(i)]
\item The complete mixability is a property of the marginal distribution $F$. For a general distribution $F$, it may or may not be CM; thus a complete mix might not exist in some cases.  On the other hand, for $F$ with finite mean, there always exists an END sequence by Theorem \ref{main1a}.
\item The END scenario allows the existence of a sequence of $X_i$, $i\in \N$ which has a global negative dependence, while the complete mix has a negative dependence only for a fixed $n$. For example, we know the uniform distribution U$[0,1]$ is $n$-CM for any $n\ge 2$. However, it is impossible to construct a sequence $X_i$, $i\in \N$ such that $\var(S_n)=0$ for all $n\ge 2$, since $\var(S_n)=0$ implies that $\var(S_{n+1})>0$. Hence, the complete mixability does not directly apply to negatively dependent sequences.
\item By Theorem \ref{main1a}, we find the END sequence by construction.  However, as pointed out in \cite{WW11}, even when the complete mixability of $F$ is shown, it remains often unclear  to construct a complete mix with marginal distribution $F$ (this is one of the open questions in complete mixability).
\end{enumerate}

In the following we connect the concepts of CM and END in a simple way.

\begin{proposition}
Suppose $F$ is $n$-CM, for some $n\in \N$ and $(Y_1,\ldots,Y_n)$ is a complete mix with marginal distribution $F$. Then the sequence of random variables $X_i,~i\in \N$ where $X_{i+(k-1)n}=Y_i$ for $i=1,\ldots,n$ and $k\in \N$ is END. Moreoever, if the variance of $F$ is finite, then \textbf{(c1)} also holds, and in addition, if $n=2$, then the sequence $X_i,~i\in \N$ is SEND.
\end{proposition}
\begin{proof}
It is  by definition that $Y_1+\cdots+Y_n=n\mu.$ We can see that by the construction of $X_i,~i\in \N$,  for any $m \in \N$, \begin{equation}|S_m-m\mu|\le \max_{i=1,\ldots,n}|Y_1+\cdots+Y_i-i\mu|.\label{CMEND}\end{equation} Hence, $S_m-m\mu$ is controlled by a random variable, which leads to the END property. It is obvious that when $F$ has a finite variance, \textbf{(c1)} holds. When $n=2$, the right-hand side of \eqref{CMEND} is either $|X_1-\mu|$ or $0$, and hence $\sup_{m\in \N}\var(S_m)\le \var(X_1)$, leading to the SEND property.
\end{proof}
Since only symmetric distributions are 2-CM (\cite{WW11}, Proposition 2.3), using CM distributions to find SEND sequences for a general $F$ may not be possible. Also, it is known to be challenging to prove complete mixability for any non-trivial class of distributions. On the other hand, our results such as Theorem \ref{main1a} do not require any additional information on the marginal distribution other than a finite mean.

On the other hand, to establish complete mixability from the END sequence is also not easy. Recall that in Theorem \ref{main1a},
$$S_n-n\mu=\left(W_1-W_2\right) (\id_{\{U\ge 1-\underline{nu(Y)}\}}-\underline{nu(Y)}),$$
where $u(Y)=(\mu-W_2)/(W_1-W_2).$
Note that for a distribution $F$ with no probability mass at $\mu$, $W_1-W_2>0$ a.s. Hence,  for $S_n$ to be a.s. a constant for a fixed $n\in \N$, one typically needs $\underline{nu(Y)}=0$ a.s. However, this requires $n(\mu-W_2)/(W_1-W_2)$ to be a.s. an integer, which is only satisfied by very specific cases of distributions $F$. One example is the symmtric distributions, where by symmetry $\mu-W_1=W_2-\mu$ a.s. and $u(Y)=1/2$ a.s. The above arguments impliy that symmtric distributions are $n$-CM for any even number $n$. This is one of the first few straightforward examples given in the theory of complete mixability (see Proposition 2.3 of \cite{WW11}). One can get similar results for the cases $u(Y)=1/k$ a.s. for some $k\in \N$, depending on different conditions of the symmetry of $F$.

\section{Applications in Risk Aggregation}

In quantitative risk management, when the marginal distributions of $X_1,\ldots,X_n$ are known but the joint distribution is unknown, risk regulators and  managers are interested in the extreme values for quantities related to an aggregate position $S_n=X_1+\cdots+X_n$ such as risk measures of $S_n$. In this section, we apply our main results to the extreme scenarios in risk management with dependence uncertainty.

\subsection{Risk aggregation with dependence uncertainty}

In the framework of risk aggregation with dependence uncertainty,
it is considered that for each $i=1,\ldots,n$ the distribution of $X_i$ is known while the joint distribution of $\mathbf{X}:=(X_1,X_2,\ldots,X_n)$ is unknown. Such setting is  practical in quantitative risk management, as statistical modeling for the dependence structure (copula) is extremely difficult especially when $n$ is relatively large. The interested reader is referred to \cite{EPR13},  \cite{BJW13} and the references therein for research in this field.  When the dependence structure is unknown, an aggregate risk $S_n$ lives in an {admissible risk class} as defined  below.
\begin{definition}[\cite{BJW13}]
The \emph{admissible risk class} is defined by the set of sums of random variables with given marginal distributions: \begin{eqnarray*}\mathfrak
S_n(F_1,\ldots,F_n)&=& \left\{X_1+\cdots+X_n: X_i\sim F_i,~i=1,\ldots,n\right\}.\end{eqnarray*}
\end{definition}

For simplicity, throughout this section, we denote by $\mathfrak S_n=\mathfrak S_n(F,\ldots,F)$.
It is immediate that the study of $\mathfrak S_n$ is equivalent to the study of question \textbf{(A)} as mentioned in the introduction.  In practice however, from a risk management perspective,   extremal problems like question \textbf{(B)}  are often of more interest.

 The following  corollary is a straightforward consequence of Theorems \ref{main1a} and \ref{main3}.
\begin{corollary}\label{coro2} Suppose $F$ is any distribution.
\begin{enumerate}[(a)]
 \item If the support of $F$ is contained in $[a,b]$, $a<b, ~a,b\in \R$, then
$$\max_{S\in\mathfrak S_n}\p\left(|S-\E[S]|\le b-a\right)=1.$$
 \item If $F$ has finite $m$-th moment, $m>1$, then uniformly in $n\in \N$, as $k\rightarrow \infty$,
$$\sup_{S\in\mathfrak S_n}\p\left(|S-\E[S]|>k\right)= o(k^{-(m-1)}).$$
\end{enumerate}
\end{corollary}
In the next two sections, we will look at the extremal  questions related to $\mathfrak S_n$.

\subsection{Bounds on convex functions and convex risk measures}
Convex order (see for example, \cite{SS07}, Chapter 1) describes the preference between risks from the perspective of risk-avoiding investors. As a classic result in this field,  the convex ordering maximum element in $\mathfrak S_n$ is always obtained by the comonotonic scenario; see \cite{DDGKV02} and \cite{DDV11} for general discussions on comonotonicity and its relevance for finance and insurance. On the other hand, finding the convex ordering minimum element for admissible risks is known to be challenging and only limited results are available; see \cite{BJW13}. For example, the infimum  on $\E[g(S)]$  over $S\in\mathfrak S_n$ for a convex function $g$ has been obtained in \cite{WW11} for marginal distributions with a monotone density and \cite{BJW13} for distributions satisfying a condition of complete mixability.

Note that for all $S\in\mathfrak S_n$, $\E[S]$ is a constant. It is well-known that $\E[g(S)]\ge g(\E[S])$ by Jensen's inequality. It is then expected that the infimum on $\E[g(S)]$ over $S\in\mathfrak S_n$ is close to the value $g(\E[S])$. If $F$ is $n$-completely mixable, the infimum is attained for the trivial case $S=\E[S]\in \mathfrak S_n$. Unfortunately, complete mixability is in general very difficult to prove, and often it is not possessed by  many distributions of practical interest. Hence, we will look at a possible upper bound for $\inf_{S\in \mathfrak S_n}\E[g(S)]$ which, along with the natural bound $g(\E[S])$, gives quite a  good estimate of $\inf_{S\in \mathfrak S_n}\E[g(S)]$.

\begin{theorem}\label{main5} Suppose $F$ is a distribution on $[a,b]$, $a<b,~a,b\in \R$,   with mean $\mu$, then for any convex function $g:\R\rightarrow \R$, $$g(n\mu)\le \inf_{S\in \mathfrak S_n}\E[g(S)]\le \frac12g(n\mu+(b-a))+\frac12g(n\mu-(b-a)).$$
\end{theorem}
\begin{proof} The first half of the inequality is due to Jensen's inequality. For the second half, by Corollary \ref{coro2},  it suffices to prove that among all distributions on $[n\mu-(b-a),n\mu+(b+a)]$ with mean $n\mu$, the Bernoulli distribution on $\{n\mu-(b-a),n\mu+(b-a)\}$ with equal probability gives the largest possible value of $\E[g(S)]$.

To show this, without loss of generality we assume $\mu=0$ with $b-a=1.$ Let $X$ be any random variable with mean 0 and support $[-1,1]$, and let $Y$ be a Bernoulli random variable with $\p(Y=1)=\p(Y=-1)=1/2$. To show that $X$ is smaller than $Y$ in convex order, it suffices to show that for each $K\in[-1,1]$, $\E[(X-K)_+]\le \E[(Y-K)_+]=(1-K)/2.$

When $\p(X\ge K)\le 1/2$, we have
$$\E[(X-K)_+]=\E[(X-K)\id_{\{X\ge K\}}]\le \E[(1-K)\id_{\{X\ge K\}}]\le \frac12 (1-K).$$
When $\p(X\ge K)> 1/2$, we have
\begin{align*}\E[(X-K)_+]=\E[(K-X)_+]+\E[X-K]&\le \E[(K-X)\id_{\{X< K\}}]-K\\&\le \E[(K+1)\id_{\{X< K\}}]-K\\&\le \frac12(K+1)-K\\&=\frac12 (1-K).\end{align*}
In conclusion, $X$ is smaller than $Y$ in convex order. Thus, $\E[g(X)]\le \E[g(Y)]$ for any convex function $g:\R\rightarrow \R$.
\end{proof}

\begin{remark}
As a special choice of $\E[g(S)]$, the variance of a sequence of identically distributed random variables is of particular importance; see Section 2. The variance bound given in  Proposition \ref{propvar} (a) is stronger than the bound in Theorem \ref{main5} which naturally gives a bound of $(b-a)^2$ if $g(x)$ is taken as $(x-n\mu)^2$.  Other quantities of the type $\E[g(S_n)]$, used in finance and insurance, include stop-loss premiums, European option prices, expected utilities and expected $n$-period returns.
\end{remark}

Another important class of quantities to discuss is the class of risk measures. In order to deteremine capital requirements for financial regulation, various risk measures are used in practice. Since the introduction of coherent risk measures by \cite{ADEH99}, there has been extensive research on coherent as well as non-coherent risk measures; see \cite{MFE05}. Two commonly used capital requirement principles are the Value-at-Risk, defined as \begin{equation}\label{varDEF}\VaR_p(X)= \inf\{x: \p(X\le x)\ge  p\},~~~~ p\in(0,1).\end{equation}
and the Expected Shortfall (ES), also known as the Tail Value-at-Risk (TVaR),
defined as \begin{equation}\label{tvarDEF}\TVaR_p(S)=\frac{1}{1-p}\int_p^1 \mathrm{VaR}_\alpha(S)\d \alpha,~~~~p\in[0,1).\end{equation}
In the case of risk aggregation with dependence uncertainty, finding bounds for VaR and ES becomes an important task (see for example \cite{EPR13}).  We will discuss the VaR case in the next section, and focus on ES for the moment.
By the subadditivity of ES, the upper sharp bound $\sup_{S\in \mathfrak S_n}\TVaR_p(S)$ for any $p\in [0,1)$ is obtained with the comonotonic  scenario, with $\sup_{S\in \mathfrak S_n}\TVaR_p(S)=n\TVaR_p(X)$ for $X\sim F$.  On the other hand, finding the explicit minimal $\TVaR$ for general marginal distributions is an open question (see \cite{BJW13} for a summary of research on explicit lower bounds in convex order for risk aggregation with dependence uncertainty).
Since the risk measure ES preserves the convex order, we have the following corollary for the smallest possible ES.

\begin{corollary}\label{coro3}
\begin{enumerate}[(a)]
\item Suppose $F$ is a distribution on $[a,b]$, $a<b, ~a,b\in \R$ with mean $\mu$, then for $p\in (0,1)$, $$n\mu\le \inf_{S\in \mathfrak S_n}\TVaR_p(S)\le n\mu+(b-a).$$
\item  Suppose $F$ is a distribution with mean $\mu$ and finite second moment, then for $p\in (0,1)$, $$n\mu\le \inf_{S\in \mathfrak S_n}\TVaR_p(S)\le n\mu+K,$$ for some constant $K$ that does not depend on $n$ but possibly depends on $p$.
\end{enumerate}
\end{corollary}
\begin{proof} Note that
$\inf_{S\in \mathfrak S_n}\TVaR_p(S)\ge \inf_{S\in \mathfrak S_n}\E[S]= n\mu$.
The other half of  part (a) comes directly from Corollary \ref{coro2}. For part (b), by Theorem \ref{main1a}, we have  that
$$\inf_{S\in \mathfrak S_n}\TVaR_p(S)=\inf_{S\in \mathfrak S_n}\TVaR_p(S_n-n\mu)+n\mu\le \TVaR_p(Z)+n\mu,$$
where $Z \sim \tilde F$. Since $F$ has finite second moment, $Z$ has finite mean, and therefore $\TVaR_p(Z)$ is finite and does not depend on $n$. This completes the proof.
\end{proof}

\begin{remark} Corollary \ref{coro3} gives estimates for the smallest possible $\TVaR_p(S)$ with dependence uncertainty. When $n$ is large and $\mu\ne 0$, the estimation errors are small compared to the major term $n\mu.$
Similar arguments will give asymptotic estimates for any convex risk measure.
\end{remark}

\subsection{Bounds on Value-at-Risk}

The popular quantile-based risk measure VaR is not a convex or coherent risk measure, hence a separate discussion is necessary. Both the maximum and the minimum of VaR with dependence uncertainty are in general unavailable analytically. For existing results on special cases of marginal assumptions, the reader is referred to the recent papers \cite{WPY13} and \cite{PR13}. For a general discussion on the bounds on VaR aggregation and numerical approximations, see \cite{EPR13}.

Recall that $F^{-1}(p)=\inf\{x: F(x)\ge p\}$ for $p\in (0,1]$, hence $\VaR_p(X)=F^{-1}(p)$ for $p\in (0,1)$ where $X\sim F$. For $1\ge q>p\ge 0$, let
$$
\mu_{p,q}=\frac{1}{q-p}\int_p^q F^{-1}(t)\d t.
$$
If $F$ is continuous, $\mu_{p,q}$ is the mean of the conditional distribution of $F$ on $[F^{-1}(p),F^{-1}(q)]$. Note that $\mu_{0,q}$ and $\mu_{p,1}$ might be infinite. 

\begin{theorem}\label{thmvar}
We have for $p\in (0,1)$ and any distribution $F$, \begin{equation}\label{vareq1a} n\mu_{p,q}-(F^{-1}(q)-F^{-1}(p))\le \sup_{S\in \mathfrak S_n}\VaR_p(S)\le n\mu_{p,1},\end{equation} for any $q\in (p,1]$, and \begin{equation}\label{vareq1b} n\mu_{0,p}\le \inf_{S\in \mathfrak S_n}\VaR_p(S)\le n\mu_{q,p}+(F^{-1}(p)-F^{-1}(q))\end{equation} for any $q\in [0,p)$.

In particular, if $F$ is a distribution on $[a,b]$, $a,b\in \R$, then for $p\in (0,1)$, $$n\mu_{p,1}-(b-F^{-1}(p))\le \sup_{S\in \mathfrak S_n}\VaR_p(S)\le n\mu_{p,1},$$ and $$n\mu_{0,p}\le \inf_{S\in \mathfrak S_n}\VaR_p(S)\le n\mu_{0,p}+(F^{-1}(p)-a).$$

\end{theorem}
\begin{proof}
First, we assume the distribution $F$ is continuous.  We will use the following equivalence lemma. A proof can be found in Section 4 of \cite{BJW13}, where the alternative definition of VaR will be used:
$$\VaR^*_p(X)=\inf \{x\in \R:\p(X\le x)>p \}.$$
\begin{lemma}
[Lemma 4.3 of \cite{BJW13}] For $p\in(0,1)$ and a continuous distribution $F$, $$\sup_{S\in \mathfrak S_n}\VaR^*_p(S)=\sup\{\mathrm{essinf} S: S\in \mathfrak S_n(F_{p},\ldots,F_{p})\},$$ and $$\inf_{S\in \mathfrak S_n}\VaR_p(S)=\inf \{\mathrm{esssup} S: S\in \mathfrak S_n(F^p,\ldots,F^p)\},$$
where $F_{p}$ is the conditional distribution of $F$ on $[F^{-1}(p),\infty)$ (upper tail), $F^{p}$ is the conditional distribution of $F$ on $(-\infty, F^{-1}(p))$ (lower tail),  $$\mathrm{esssup} S=\sup\{t\in \R: \p(S\le t)<1\},$$ and  $$\mathrm{essinf} S=\inf\{t\in \R : \p(S\le t)>0\}.$$
\end{lemma}
Note the asymmetry between the supremum and infimum.
We first show that \begin{equation}\label{vareq1}\sup_{S\in \mathfrak S_n}\VaR^*_p(S)=\sup\{\mathrm{essinf} S: S\in \mathfrak S_n(F_{p},\ldots,F_{p})\}\ge n\mu_{p,q}-(F^{-1}(q)-F^{-1}(p))\end{equation} for $0<p<q\le 1$. The case when $F^{-1}(q)=\infty$ is trivial, hence we only consider the case when $F^{-1}(q)<\infty.$

Let $F_{p,q}$ be the conditional distribution of $F$ on $[F^{-1}(p),F^{-1}(q)]$ for $0<p<q\le 1$. By Corollary \ref{coro1}, there exist random variables $X_1,\ldots,X_n$ from $F_{p,q}$ such that $X_1+\cdots+X_n \ge m\mu_{p,q}-(F^{-1}(q)-F^{-1}(p)).$ Let $Z$ be any random variable with distribution $F_q$ and let $C$ be a random event independent of $X_1,\ldots,X_n,Z$, with $\p(C)=(q-p)/(1-p).$ Define $Y_i=X_i\id_C+Z(1-\id_C)$ for $i=1,\ldots,n.$
It is straightforward to check that $Y_i$ has distribution $F_p$, and $$Y_1+\cdots+Y_n\ge X_1+\cdots+X_n \ge n\mu_{p,q}-(F^{-1}(q)-F^{-1}(p)).$$ Thus, $\mathrm{essinf}(Y_1+\cdots+Y_n)\ge  n\mu_{p,q}-(F^{-1}(q)-F^{-1}(p)),$ and we obtain \eqref{vareq1}. Since $\VaR_p(X)\ge \VaR^*_{r}(X)$ for any $r<p$ and random variable $X$,  we have that
\begin{align*}\sup_{S\in\mathfrak S_n}\VaR_p(S_n)\ge \lim_{r\rightarrow p-}\sup_{S\in\mathfrak S_n}\VaR^*_{r}(S_n)&\ge  \lim_{r\rightarrow p-}(n\mu_{r,q}-(F^{-1}(q)-F^{-1}(r)))\\&= n\mu_{p,q}-(F^{-1}(q)-F^{-1}(p)).\end{align*}
Note that here we use the fact that $F^{-1}$ is left-continuous.  On the other hand,  $$\sup_{S\in\mathfrak S_n}\VaR_p(S)\le \sup_{S\in\mathfrak S_n}\VaR_p^*(S)= \sup\{\mathrm{essinf} S: S\in \mathfrak S_n(F_{p},\ldots,F_{p})\}\le n\mu_{p,1}$$ always holds trivially. Thus we obtian \eqref{vareq1a} for continuous distribution $F$.

If $F$ is not continuous, let $X\sim F$ and $U_\epsilon\sim \mathrm{U}[0,\epsilon]$, $\epsilon>0$, be independent of $X$. Denote $X_\epsilon=X-U_\epsilon$ and let $F_\epsilon$ be the distribution function of $X_\epsilon$. It is easy to check that  $F_\epsilon$ is a continuous distribution. By the monotonicity of VaR, it is easy to check that $$\sup \{\VaR(S):S\in \mathfrak S_n(F,\cdots,F)\}\ge \sup \{\VaR(S):S\in \mathfrak S_n(F_\epsilon,\cdots,F_\epsilon)\}.$$
For $1\ge q>p\ge 0$, let
$$\mu^{\epsilon}_{p,q}=\frac{1}{q-p}\int_p^q F_\epsilon^{-1}(t)\d t.$$

Since $X_\epsilon\rightarrow X$ in $L^\infty$, it is easy to see that $F_\epsilon ^{-1}(t)\rightarrow F^{-1}(t)$ for $t\in [0,1]$ as $\epsilon \rightarrow 0.$ By Fatou's lemma, we also have $\mu^{\epsilon}_{p,q}\rightarrow \mu_{p,q}$ as $\epsilon \rightarrow 0$. It follows from \eqref{vareq1a} for $F_\epsilon$ that
\begin{align*}
 \sup_{S\in \mathfrak S_n}\VaR_p(S)&\ge  \sup_{S\in \mathfrak S_n(F_\epsilon,\ldots,F_\epsilon)}\VaR_p(S) \\&\ge n\mu^\epsilon_{p,q}-(F_\epsilon^{-1}(q)-F_\epsilon^{-1}(p))\\&\rightarrow  n\mu_{p,q}-(F^{-1}(q)-F^{-1}(p)), \mbox{~as $\epsilon\rightarrow 0$.}
\end{align*}
Thus, we complete \eqref{vareq1a}. Similarly, we can show \eqref{vareq1b}.
\end{proof}

\begin{corollary}\label{estvar} Suppose $F$ has finite $k$-th moment, $k\ge 1$. Then
$${\sup_{S\in\mathfrak S_n}\VaR_p(S)}=n\mu_{p,1}-o(n^{1/k}),$$
and
$${\inf_{S\in\mathfrak S_n}\VaR_p(S)}=n\mu_{0,p}+o(n^{1/k}).$$
\end{corollary}
\begin{proof} Without loss of generality we assume $F^{-1}(p)\ge 0$ (otherwise this assumption can easily be satisfied with a shift of location).
Choose $q_n=F(an^{1/k})$ in \eqref{vareq1a} for any constant $a>0$ and large $n$ such that $q_n>p$. We have that $F^{-1}(q_n)\le an^{1/k}$, and \begin{align}\nonumber \sup_{S\in\mathfrak S_n}\VaR_p(S)\ge   n\mu_{p,q_n}-(F^{-1}(q_n)-F^{-1}(p)) &\ge  n\mu_{p,q_n}- an^{1/k} \\&=n\mu_{p,1}-n(\mu_{p,1}-\mu_{p,q_n})-an^{1/k}.\label{estvar1}\end{align}
Note that for $X\sim F$, $$\mu_{p,1}-\mu_{p,q_n} =\frac{1}{1-p}\E[X\id_{\{X\ge F^{-1}(p)\}}]-\frac{1}{q_n-p}\E[X\id_{\{F^{-1}(q_n)\ge X\ge F^{-1}(p)\}}]\le \frac{1}{1-p}\E[X\id_{\{X\ge F^{-1}(q_n)\}}].$$
Since $F$ has finite $k$-th moment, we have that
$$(an^{1/k})^{(k-1)}\E[X\id_{\{X\ge an^{1/k}\}}]\le\E[|X|^k\id_{\{X\ge an^{1/k}\}}]\rightarrow 0,$$
 and hence $\E[X\id_{\{X\ge an^{1/k}\}}]= o(n^{-1+1/k}) $. Thus, $\frac{1}{1-p}\E[X\id_{\{X\ge F^{-1}(q_n)\}}]=o(n^{-1+1/k}),$ which, together with \eqref{estvar1}, leads to
$${\sup_{S\in\mathfrak S_n}\VaR_p(S)}\ge n\mu_{p,1}-o(n^{1/k})-an^{1/k}.$$
Since $a>0$ is arbitrary, and $\sup_{S\in\mathfrak S_n}\VaR_p(S)\le  n\mu_{p,1}$, we have that
$${\sup_{S\in\mathfrak S_n}\VaR_p(S)}= n\mu_{p,1}-o(n^{1/k}).$$
The other half of the corollary is obtained similarly.
\end{proof}

\begin{remark}Theorem \ref{thmvar} and Corollary \ref{estvar} provide quite good estimates for the worst-case (best-case) VaR   under dependence uncertainty. The estimation becomes accurate when $n$ is large, as $n\mu_{p,1}$ (or $n\mu_{0,p}$) is large compared to the estimation error which is controlled within a rate of $n^{1/k}$, except for the trivial case when $\mu_{p,1}=0$ (or $\mu_{0,p}=0$).
\end{remark}

The next two corollaries give the asymptotic limit of the superadditive ratio (see \cite{EPR13}) for VaR and the asymptotic equivalence between worst-case VaR and worst-case ES (see \cite{PWW13}).

\begin{corollary}\label{equiva2} For any distribution $F$,  as $n\rightarrow \infty$,
$$\frac{\sup_{S\in\mathfrak S_n}\VaR_p(S)}{n }\rightarrow  {\TVaR_p(X)} $$
where $X\sim F$.
\end{corollary}
\begin{proof}
We take $q_n=F(\sqrt n)$ for large $n$ such that $q_n>p$. It follows from \eqref{vareq1a} that
$$\mu_{p,q_n}-o\left(\frac{1}{\sqrt n}\right)\le \frac{\sup_{S\in\mathfrak S_n}\VaR_p(S)}{n}\le \mu_{p,1}.$$ Obviously $q_n\rightarrow 1$, and hence $\mu_{p,q_n}\rightarrow \mu_{p,1}.$  This completes the proof.
In fact, if $\mu_{p,1}<\infty$, then $F$ has finite mean and
Corollary \ref{equiva2} follows directly from Corollary \ref{estvar} by taking $k=1$.
\end{proof}

\begin{remark}
The fraction $ {\sup_{S\in\mathfrak S_n}\VaR_p(S)}/{(n\VaR_p(X))}$ for $\VaR_p(X)>0$ is called the (worst) superadditive ratio of VaR (see \cite{EPR13}). It measures the amount of  possible extra capital requirement needed in a diversification strategy, and hence this quantity is of independent interest in quantitative risk management.
 Corollary \ref{equiva2} gives the limit as ${\TVaR_p(X)}/{\VaR_p(X)}$ without assuming any condition on $F$. Note that here $\TVaR_p(X)$ can be infinite. Hence, whenever $\TVaR_p(X)=\infty$, the superadditive ratio of VaR becomes infinity. This fact clearly shows that the ``diversification benefits" commonly used in practical risk management needs to be taken with care.
\end{remark}

\begin{corollary}\label{equiva} Suppose $F$ has finite $k$-th moment, $k\ge1$ and non-zero $\TVaR$ at level $p\in (0,1)$, then  as $n\rightarrow \infty$,
\begin{equation}\label{limeq2}\frac{\sup_{S\in\mathfrak S_n}\VaR_p(S)}{\sup_{S\in\mathfrak S_n}\TVaR_p(S)}=1-o(n^{-1+1/k}).\end{equation}
\end{corollary}
\begin{proof}
Note that $\sup_{S\in\mathfrak S_n}\TVaR_p(S)=n\TVaR_p(X)=n\mu_{p,1}\ne 0$ for $X\sim F$. Thus, the proof follows directly from Corollary \ref{estvar}.
\end{proof}

\begin{remark} Corollary \ref{equiva} implies that
 under the worst-case scenario of dependence, the VaR and ES risk measures are asymptotically equivalent; that is,
\begin{equation}\label{limeq}\frac{\sup_{S\in\mathfrak S_n}\VaR_p(S)}{\sup_{S\in\mathfrak S_n}\TVaR_p(S)}\rightarrow 1\end{equation}
This coincides with the main results in \cite{PR13d} and \cite{PWW13}. \cite{PR13d} obtained \eqref{limeq} under a condition of complete mixability, which at this moment is only known to be  satisfied by tail-monotone densities. \cite{PWW13} gave \eqref{limeq} under a weaker condition that $F$ has a strictly positive density and discussed some possible inhomogeneous cases.  Both of the above papers assumed the continuity of $F$. To ensure \eqref{limeq}, Corollary \ref{equiva} only assumes that $\TVaR_p(X_1)$ is finite and non-zero, which is necessary. Hence, Corollary \ref{equiva}  establishes the weakest mathematical assumption for \eqref{limeq} to hold. In addition, Corollary \ref{equiva} also gives the convergence rate of this asymptotic equivalence. We can see that the convergence in \eqref{limeq2} is fast for the distribution $F$  being light-tailed,  and slow for $F$ being heavy-tailed. This gives a theoretical justification of the discussion on the numerical illustrations in Section 5 of \cite{PWW13}, where it was observed that heavy-tailed marginal distributions in general lead to a slower convergence of \eqref{limeq}, compared to the cases of light-tailed marginal distributions.
\end{remark}

 \section{Conclusion}

In this paper, we introduce the notions of extreme negative dependence (END) and strong extreme negative dependence (SEND) scenarios, and showed that for each marginal distribution $F$ with finite mean, a construction of an END sequence is always possible.  With a finite third moment of $F$, an SEND sequence is also obtained by the same construction. The sum of END random variables is in general concentrated around its expectation, and the difference $|S_n-n\mu|$ is controlled by a random variable that does not depend on $n$. We suggest that the concept of END,  comparable to the concepts of independence and comonotonicity, is a new benchmark for negative correlation in the study of summation of random variables.
We also studied asymptotic bounds for risk aggregation with dependence uncertainty and provided estimates for the worst-case and best-case risk measures VaR and TVaR.

The concepts of END and SEND can naturally be generalized to the case of an inhomogeneous (non-identically distributed) sequence of random variables, leading to a potential direction of future research.   Generalizations of END and SEND in a multi-dimensional setting is also a promising research direction especially with  applications in QMC simulation.

\subsection*{Acknowledgement} R. Wang acknowledges support from the Natural Sciences
and Engineering Research Council of Canada (NSERC) and the Forschungsinstitut f\"ur Mathematik (FIM) at ETH Zurich during his visit in Zurich.
The authors would also like to thank Paul Embrechts (ETH Zurich), Thomas Mikosch (U Copenhagen) and Giovanni Puccetti (U Firenze) for helpful comments which have essentially improved the paper.

\end{document}